\documentclass{article}
\usepackage[english]{babel}
\usepackage{amsmath,amssymb,amsthm}
\usepackage{graphicx, bm}
\usepackage{color}
\usepackage{multirow}
\usepackage{todonotes} 

\newtheorem{thm}{Theorem}
\newtheorem{lem}[thm]{Lemma}
\newtheorem{rmk}{Remark}

\begin{document}

\title{A nonlinear aggregation type classifier}

\author{Alejandro Cholaquidis, Ricardo Fraiman, Juan Kalemkerian\\ and Pamela Llop}
\maketitle

\begin{abstract}
We introduce a nonlinear aggregation type classifier for
 functional data defined on a separable and complete metric space.
 The new rule is built up from a collection of  $M$ arbitrary
 training classifiers. If the classifiers are consistent, then  so is the aggregation rule. Moreover, asymptotically the aggregation rule behaves as well as the best of the $M$
 classifiers. The results of a small si\-mu\-lation are reported both, for high dimensional
 and functional data, and a real data example is analyzed. 
\end{abstract}


 \section{Introduction} 
 Supervised classification is still one of the  hot topics for high dimensional and functional data 
due to the importance of their applications and the intrinsic difficulty in a general setup. 
In this context, there is a vast literature on classification methods which include: 
linear classification, $k$-nearest neighbors and kernel rules, classification based on partial 
least squares, reproducing kernels or depth measures. Complete surveys of the literature are the 
works by Ba\'illo et al. \cite{baillo_2011}, Cuevas \cite{cuevas_2012} and Delaigle and Hall \cite{hall_2012}. In 
the book \textit{Contributions in infinite-dimensional statistics and related topics} \cite{Bong_2014}, there 
are also several recent advances in supervised and unsupervised classification. See for instance, 
Chapters 2, 5, 22 or 48, or directly, Chapter 1 of this issue (Bongiorno et al. \cite{Bong1_2014}). In this context, very recently there have been of great interest to develop aggregation methods. In particular, there is a large list of linear aggregation methods like boosting (Breiman \cite{breiman_96}, Breiman \cite{breiman_98}), random forest (Breiman \cite{breiman_2001}, Biau et al. \cite{biau_2008},  Biau \cite{biau_2012}), among others. All these  methods exhibit an important improvement when combining a subset of  classifiers to produce a new one. Most of the contributions to the aggregation literature have been proposed for nonparametric regression, a problem closely related to classification rules,  
which can be obtained just by plugging  in the estimate of the regression function into the Bayes rule (see for instance, Yang \cite{yang_2004} and Bunea et al. \cite{bunea_2007}). Model  selection (select the optimal single model from a list of models), convex aggregation (search for the optimal convex combination of a given set of estimators), and linear aggregation (select the optimal linear combination of estimators) are important contributions among a large list.

In the finite dimensional setup, Mojirsheibani  \cite{mojir1} and  \cite{mojir2} introduced a 
combined classifier showing strong consistency under someway hard to verify assumptions involving 
the Vapnik Chervonenkis dimension of the random partitions of the set of classifiers, which are  
non--valid in the functional setup. Very recently Biau et al. \cite{biau_2013} introduced a 
new nonlinear aggregation strategy for the regression problem called COBRA, extending the ideas in 
Mojirsheibani  \cite{mojir1} to the more general setup of nonparametric regression in $\mathbb R^d$. In the 
same direction but for the classification problem in the infinite dimensional setup, we 
extend the ideas in Mojirsheibani \cite{mojir1} to construct a classification rule which combines, in a 
nonlinear way, several classifiers to construct an optimal one. We point out that our rule 
allows to combine methods of very different nature, taking advantage of the abilities of each 
expert and allowing to adapt the method to different class of datasets. Even though our 
classifier allows aggregate experts of the same nature, the  possibility of combine classifiers of 
different character, improves the use of existing rules as the bagged nearest neighbors classifier 
(see for instance Hall and Samworth \cite{hall_2005}). As in Biau et al. \cite{biau_2013}, we also introduce a more flexible form of the rule which discards a small percentage $\alpha$ of those preliminary experts that behaves differently from the rest. Under very mild assumptions, we prove consistency, obtain rates of convergence and show some optimality properties of the aggregated rule. To build up this classifier, we use the inverse function (see also Fraiman et al. \cite{fraiman_etal_97}) of each preliminary experts which makes the proposal particularly well designed for high dimensional data avoiding the curse of dimensionality. It also performs well in functional data settings. 
  
In Section \ref{setup} we introduce the new  classifier in the general context of a separable and 
complete metric space which combines, in a nonlinear way, the decision of $M$ experts (classifiers). 
A more flexible rule is also considered. In Section \ref{resultados} we state our two main results 
regarding consistency, rates of convergence and asymptotic optimality of the classifier. 
Asymptotically, the new rule performs as the best of the $M$ classifiers used to build it up. 
Section \ref{simus} is devoted to show through some simulations the performance of the new 
classifier in high dimensional and functional data for moderate sample sizes.  A real data example is also considered. All proofs are given 
in the Appendix. 
  
\section{The setup}\label{setup}

Throughout the manuscript $\mathcal{F}$ will denote a separable and complete metric space, $(X,Y)$ a
random pair taking values in $\mathcal{F} \times \{0,1\}$ and $\mu$  the probability measure of
$X$. The elements of the training sample  $\mathcal{D}_n\hspace{-0.1cm}=\hspace{-0.1cm}\{(X_1,Y_1),\dots,(X_n,Y_n)\}$, are iid
random elements with the same distribution as the pair $(X,Y)$.  The regression function is denoted by $\eta(x)= \mathbb{E}(Y|X=x) = \mathbb P(Y=1|X=x)$, the Bayes rule by 
 $g^*(x)=\mathbb{I}_{\{\eta(x)>1/2\}}$ and the optimal Bayes risk by $L^*=\mathbb{P}\big(g^*(X)\neq Y\big)$.

In order to define our classifier, we split the sample $\mathcal{D}_n$ into two subsamples
$\mathcal{D}_k=\big\{(X_1,Y_1),\dots,(X_k,Y_k)\big\}$ and
$\mathcal{E}_l=\big\{(X_{k+1},Y_{k+1}),\dots,(X_n,Y_n)\big\}$ with $l=n-k\geq 1$. With
$\mathcal{D}_k$ we build up $M$ classifiers $g_{mk}:\mathcal{F}\rightarrow \{0,1\}$, $m=1,\dots,M$
which we place in the vector $\mathbf{g_k}(x) \doteq \big(g_{1k}(x),\dots,g_{Mk}(x)\big)$ and,
following some ideas in \cite{mojir1}, with $\mathcal{E}_l$ we construct
our aggregate classifier as,
\begin{equation}\label{clasificador}
g_T(x)=\mathbb{I}_{\{T_n(\mathbf{g_k}(x))>1/2\}},
\end{equation}
where
\begin{equation}\label{agregados}
T_n(\mathbf{g_k}(x))=\sum_{j=k+1}^n W_{n,j}(x)Y_{j},\quad x\in \mathcal{F},
\end{equation}
with weights $W_{n,j}(x)$ given by
\begin{equation}\label{pesos}
W_{n,j}(x)=\frac{\mathbb{I}_{\{\mathbf{g_k}(x)=\mathbf{g_k}(X_j)\}}}{\sum_{i=k+1}^n\mathbb{I}_{\{
\mathbf{g_k}(x)=\mathbf{g_k} (X_i)\}}}.
\end{equation}
Here, $0/0$ is assumed to be $0$. Like in \cite{biau_2013}, for $0 \le \alpha < 1$ a more flexible 
version of the classifier, called $g_T(x,\alpha)$, can be defined replacing the 
weights in (\ref{pesos}) by
\begin{equation}\label{pesos2}
W_{n,j}(x)=\frac{\mathbb{I}_{\{\frac{1}{M}\sum_{m=1}^M \mathbb{I}_{\{g_{mk}(x)
= g_{mk}(X_j)\}} \ge 1-\alpha\}}}{\sum_{i=k+1}^n\mathbb{I}_{\{\frac{1}{M}\sum_{m=1}^M
\mathbb{I}_{\{g_{mk}(x) = g_{mk}(X_i)\}} \ge 1-\alpha\}}}.
\end{equation}
More precisely, the more flexible version of the classifier (\ref{clasificador}) is given by
\begin{equation}\label{clasifgen}
g_T(x,\alpha)=\mathbb{I}_{\{T_n(\mathbf{g_k}(x),\alpha)>1/2\}},
\end{equation}
where $T_n(\mathbf{g_k}(x),\alpha)$ is defined as in (\ref{agregados}) but with the 
weights given by (\ref{pesos2}).
Observe that if we choose $\alpha=0$ in (\ref{pesos2}) and (\ref{clasifgen}) we obtain the weights given in (\ref{pesos})
and the classifier (\ref{clasificador}) respectively. 
\begin{rmk}
\begin{itemize}
\item[a)] The type of nonlinear aggregation used to define our classifiers turns out
to be quite natural. Indeed, we give a weight different from zero
to those $X_j$ which classify $x$ in the same group as the whole
set of classifiers $\mathbf{g_k}(X_j)$ (or $100(1 - \alpha) \%$ of
them). 
\item[b)] Since we are using the inverse functions of the classifiers $g_{mk}$, observations which
are far from $x$  for which the condition mentioned in a) is fulfilled are involved in the
definition of the classification rule. This may be very important in the case of high dimensional
data to avoid the curse of dimensionality. This is illustrated in Figure \ref{fig:example}, where 
we show two samples of points: one uniformly distributed in the square $[-2,2]\times[-2,2]$ (filled 
black points) and another uniformly distributed in the $L_{\infty}$-ring $[-2,2]\times[-1,1]$ (empty 
black points). We also show two points to classify, the empty red and the filled magenta triangles 
together with their corresponding voters, empty green squares and filled blue 
squares, respectively. As we can see, observations 
that are far from the triangles are also involved in the classification. 
\end{itemize}
\end{rmk}

 \begin{figure}[!t]
 \begin{center}
 \includegraphics[width=0.49\textwidth]{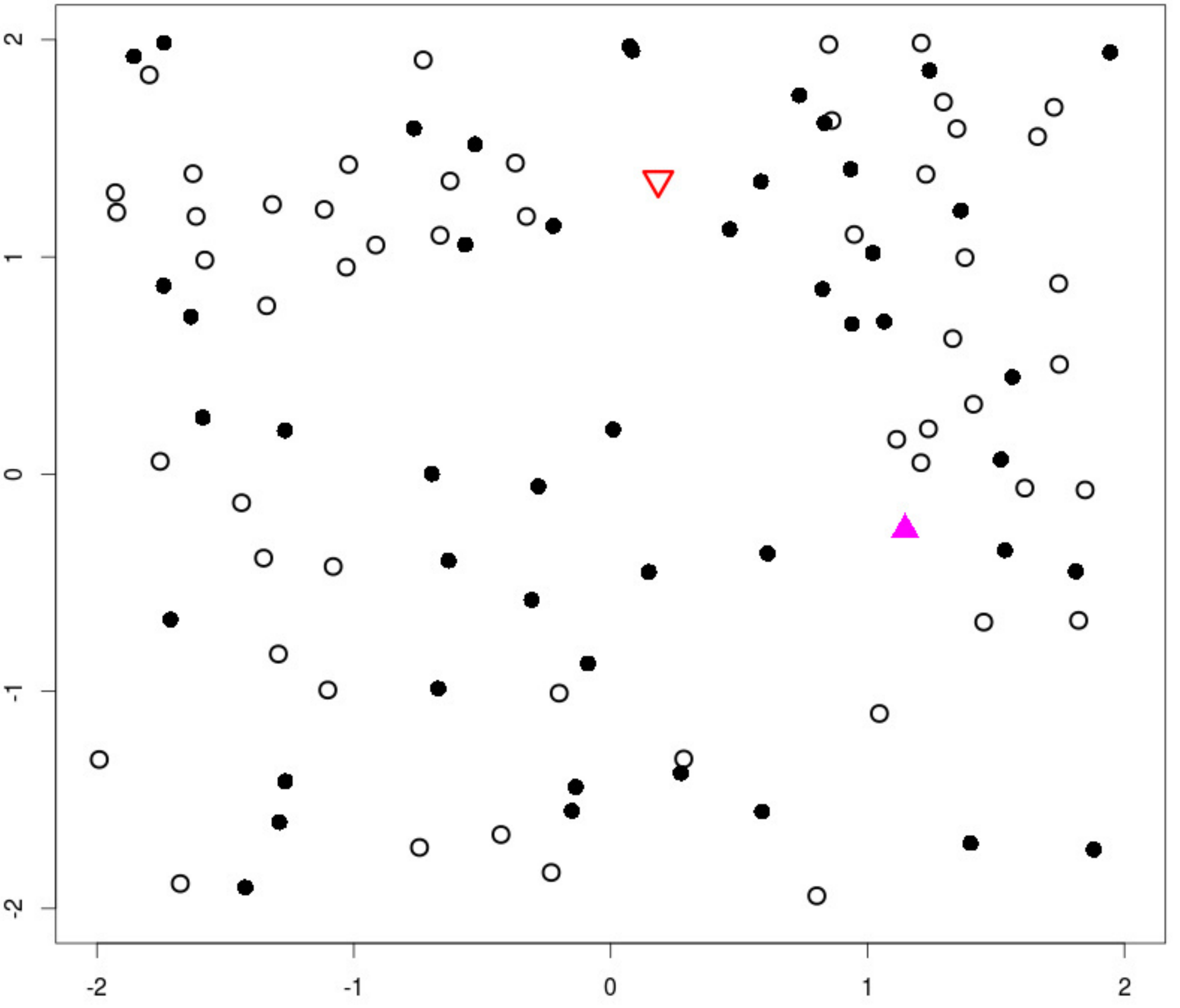}
 \includegraphics[width=0.49\textwidth]{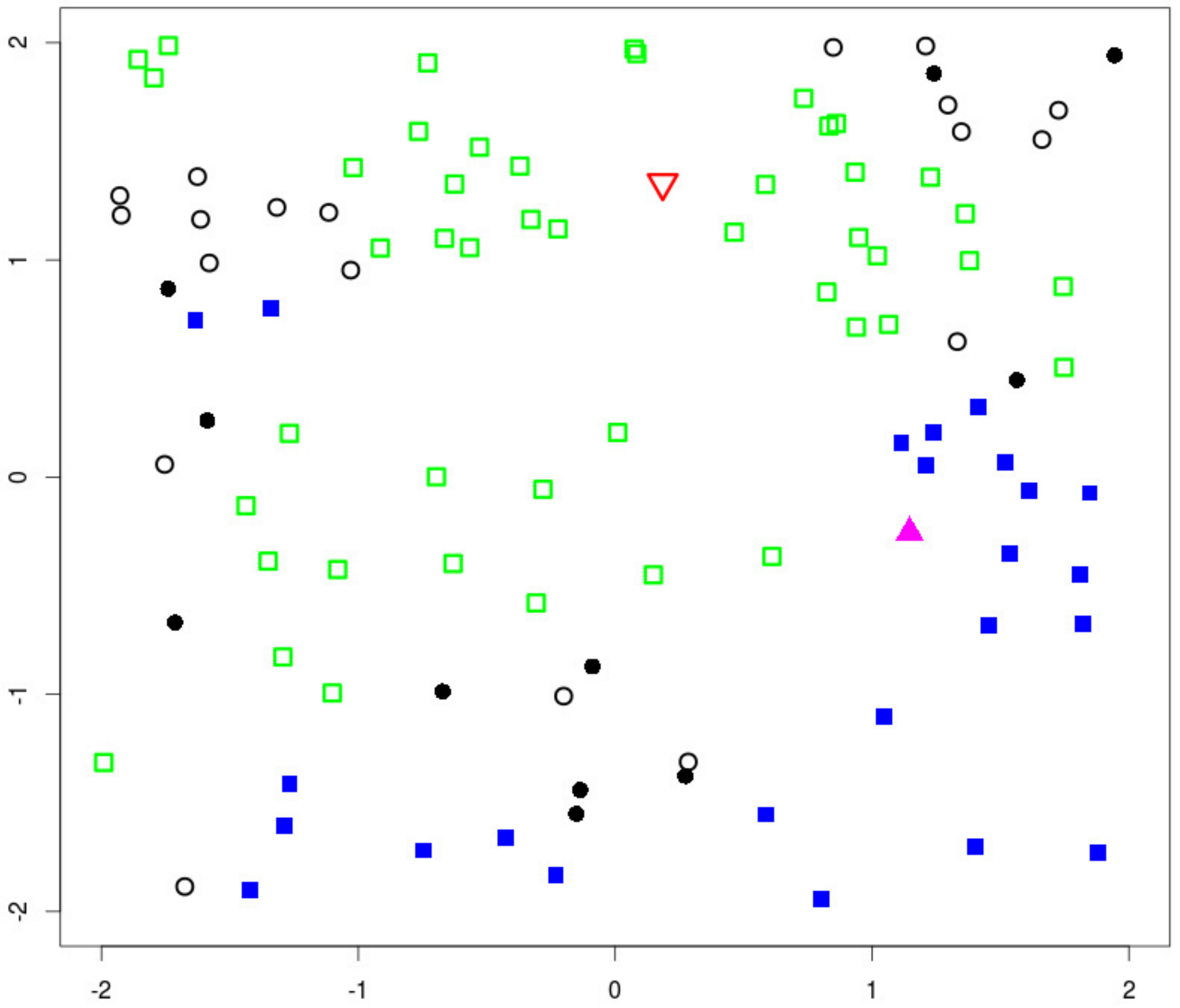}
 \end{center}
 \vspace{-0.5cm} 
 \caption{Left: Sample points corresponding to two populations (black filled and empty circles) and two points to classify (red empty and magenta filled triangles). Right: in empty green squares the voters for the red empty triangle and in filled blue squares the voters for the magenta filled triangle.}
 \label{fig:example}
 \end{figure}
	
\section{Asymptotic results}\label{resultados}
In this section we show two asymptotic results for the nonlinear aggregation classifier
%$g_T(X,\alpha)$ (which in particular include the corresponding result for $g_T(X)$). 
The first one shows that the classifier
$g_T(X,\alpha)$ is consistent if, for $0\leq \alpha <0.5$, at least $R\geq (1-\alpha)M$ of them are consistent. 
Moreover, rates of convergence for $g_T(X,\alpha)$ (and $g_T(X)$) are obtained assuming we know the rates of
convergence of the $R$ consistent experts. The second result, shows that $g_T(X)$
%(and in consequence $g_T(X)$) 
behaves asymptotically as the best of the $M$ classifiers used to build it
up. Both results are proved under mild conditions. Throughout this section we will 
use the notation $\mathbb P_{\mathcal{D}_k}(\cdot)=\mathbb P(\cdot|\mathcal{D}_k)$.

\begin{thm} \label{consistencia} \label{cor1} Assume that, for every $m=1, \ldots, R$, the classifier
$g_{mk}$
converges in probability to $g^*$ as $k \to \infty$, with $R\geq M(1-\alpha)$ and $\alpha \in [0,1/2)$. 
  Let us assume that $\mathbb P(Y=1|g^*(X)=1)>1/2$ and $\mathbb
P(Y=0|g^*(X)=0)>1/2$, then 
\begin{itemize}
\item[a)] $\displaystyle \lim_{min\{k,l\}\rightarrow \infty} \mathbb P_{\mathcal{D}_k}(g_T(X,\alpha)\neq
Y)-L^*=0.$
\item[b)] Let $\beta_{mk}\rightarrow 0$ as $k \to \infty$, for $m=1,\dots,R$ and $\displaystyle \bm{\beta_{Rk}}=\max_{m=1,\dots,R}\beta_{mk}$.
  If $ \mathbb{P}_{\mathcal{D}_k}\big(g^*(X)\neq g_{mk}(X)\big)=\mathcal{O}(\beta_{mk})$, then, for $k$ large enough,
\begin{equation} \label{ordteo1}
 \mathbb P_{\mathcal{D}_k}(g_T(X,\alpha)\neq Y)-L^*=\mathcal{O}\Big(\max\big\{\exp(-Cl),\bm{\beta_{Rk}}\big\}\Big),
\end{equation}
for some constant $C>0$. 
\end{itemize}
\end{thm}
\begin{rmk} 
\begin{itemize}
\item[a)] The assumption 
\begin{equation} \label{condteo1}
1) \ \mathbb P(Y=1|g^*(X)=1)>1/2 \qquad 2)\ \mathbb P(Y=0|g^*(X)=0)>1/2,
\end{equation} 
is really mild. It just requires that if the Bayes rule $g^*(X)$ takes the value $1$ (or 0) the probability that $Y=1$ is greater than the probability that $Y=0$ (the probability that $Y=0$ is greater than the probability that $Y=1$). Moreover since the Bayes risk $L^*\leq 1/2$ one of the conditions in (\ref{condteo1}) is always fulfilled.  
\item[b)] It is well known that in the finite dimensional case, if the regression function $\eta$ verifies a Lipschitz condition and $X$ is bounded
supported, the accuracy of classical classification rules is $\mathcal{O}(n^{-2/(d+2)})$. Therefore the right hand side of (\ref{ordteo1}) is
$$\mathcal{O}\Big(\max\big\{\exp(-Cl),k^{-2/(d+2)}\big\}\Big),$$
and the optimal rate for $\max\big\{\exp(-Cl),k^{-2/(d+2)}\big\}$ is attained for $l\sim \log(k)$.
\item[c)] The choice of the parameters $\alpha$ and $l$ is an important issue. From a practical 
point of view, we suggest to perform a cross validation procedure to select the values of the 
corresponding parameters. See Section \ref{realdata} for an implementation in a real data example.
\end{itemize}
\end{rmk}
In order to state the optimality result we introduce some additional notation. Let $\mathbb{C}
\doteq \{0,1\}^M$ and let us call $\nu \in \mathbb{C}$. Calling $\nu(m)$  the $m$-th entry of the vector $\nu$, we define the following subsets
\[
A_{\nu}^0 \doteq \bigcap_{m=1}^M g_{mk}^{-1}(\nu(m)) \times \{0\},
\hspace{0.3cm}  A_{\nu}^1 \doteq \bigcap_{m=1}^M g_{mk}^{-1}(\nu(m))
\times \{1\},
\]
\[\text{ and }  \hspace{0.1cm} A_{\nu} =A_{\nu}^0 \cup A_{\nu}^1.
\]
For each $\nu \in \mathbb{C}$, we consider the assumption:
$$
(\mathcal{H}) \quad H(\mathcal{D}_k) := \mathbb{P}_{\mathcal{D}_k}\big((X,Y) \in A_{\nu}^1\big)
- \mathbb{P}_{\mathcal{D}_k}\big((X,Y) \in A_{\nu}^0\big)\neq 0 \hspace{0.3cm} \text{a.s.}
$$

\begin{thm}\label{optimalidad}

\begin{itemize}
\item[1)] 
For each $m=1,\ldots, M$,
\[
\mathbb{P}_{\mathcal{D}_k}\big(g_T(X) \neq Y\big)  -
\mathbb{P}_{\mathcal{D}_k}\big(g_{mk}(X) \neq Y\big)\leq \mathcal{O}_k\big(l^{-1/2}\big),
\]
which implies that, 
\[
\lim_{l \to \infty} \mathbb{P}_{\mathcal{D}_k}\big(g_T(X) \neq Y\big)  \le \min_{1\le m \le M
}
\mathbb{P}_{\mathcal{D}_k}\big(g_{mk}(X) \neq Y\big).
\]
\item [2)]Under assumption ($\mathcal{H}$) we obtain a better approximation rate,
\[
\mathbb{P}_{\mathcal{D}_k}\big(g_T(X) \neq Y\big)  -
\mathbb{P}_{\mathcal{D}_k}\big(g_{mk}(X) \neq Y\big)\leq \mathcal{O}_k\big(\exp(-K_1l)\big).
\]
\end{itemize}
\end{thm}

 \section{A small simulation study}\label{simus}
 In this section we present  the performance of the aggregated classifier  in two different
 scenarios. The first one corresponds to high dimensional data while, in the second one, we
consider  two simulated models for functional data analyzed in Delaigle and Hall \cite{hall_2012}. \\ 

\subsection*{High dimensional setting}
In this setting we show the performance of our method by analyzing data ge\-ne\-ra\-ted in $ \mathbb R^{150}$ in the following way: we ge\-ne\-rate $n+200$ iid  uniform random 
variables in $[0,1]$, say $Z_1,\ldots,Z_{n+200}$.
For each $i=1,\ldots, n+200$, if $Z_i >1/4$, we generate
a random  variable $X_i \in  \mathbb R^{150}$ with uniform distribution in  $[-2,2]^{150}$ and
set $Y_i=1$. If $Z_i \le 1/4$, we generate a random variable $X_i\in \mathbb R^{150}$ with uniform
distribution in $\tau_v([-2,2]^{150})$ where $\tau_v$ is the translation along the direction
$(v,\dots,v)\in \mathbb{R}^{150}$ for $v = 1/4$ and set $Y_i =0$. Then we split the sample into two
 subsamples: with the first $n$ pairs $(X_i,Y_i)$, we build the training sample, with the
remaining $200$ we build the testing sample. We consider two cases: the 
homogeneous case, where we aggregate classifiers of the same nature and in 
the heterogeneous case, where we aggregate experts of different nature.
\begin{itemize}
\item  \underline{Homogeneous case}:  $M$ $k$-nearest neighbor classifiers  with the number of neighbors taken as follows:
\begin{enumerate}
\item\label{1} we fix $M=8$ consecutive odd numbers;
\item\label{2} we choose at random $M=10$ different odd integers between $1$ and \\
\mbox{$\min\{\sum_{i=1}^k Y_i, k-\sum_{i=1}^k Y_i\}$}.
\end{enumerate}
\end{itemize}

In Table \ref{tabla11}, we report the mean and standard deviation (in brackets) of 
the misclassification error rate for case  \ref{1}, when compared with the nearest neighbor rules 
build up with a sample size $n$
 taking  $5,7,9,11,13,15,17,19$ nearest neighbors (these classifiers are denoted by $g_{mn}$ for 
$m=1,\dots,8$). In Table \ref{tabla12} we report the median and MAD  (in brackets) 
 of the misclassification error rate for this case. 

\begin{table}[t!]
\footnotesize{ 
 \begin{center}
  \begin{tabular}{|c|c|c|c|c|c|c|c|c|c|}
 \hline
 n/k    &$g_T(\cdot)$& $g_{1n}$&$g_{2n}$&$g_{3n}$&$g_{4n}$&$g_{5n}$&$g_{6n}$&$g_{7n}$&$g_{8n}$\\
 \hline
 400/300&\textbf{.027}&.045&.043&.042&.042&.043&.043&.043&.044\\
&(.014)&(.016)&(.017)&(.017)&(.018)&(.018)&(.018)&(.019)&(.019)\\
\hline     
 600/400&\textbf{.023}&.039&.036&.035&.035&.035&.036&.037&.037\\
&(.012)&(.015)&(.016)&(.015)&(.015)&(.015)&(.015)&(.016)&(.016)\\
 \hline
 800/600&\textbf{.020}&.037&.034&.033&.033&.033&.033&.033&.033\\
&(.010)&(.014)&(.013)&(.013)&(.013)&(.013)&(.013)&(.013)&(.014)\\
 \hline
 \end{tabular}
 \end{center}}
 \vspace{-0.5cm}
 \caption{Mean and standard deviation of the misclassification error rate over $500$ replicates 
for  $\mathbb{R}^{150}$ with fixed number of neighbors.}\label{tabla11}
 \end{table}
 \begin{table}[t!]
\footnotesize{ 
 \begin{center}
  \begin{tabular}{|c|c|c|c|c|c|c|c|c|c|}
 \hline
  \hspace{-0.3cm} n/k  \hspace{-0.3cm}  &$ \hspace{-0.2cm} g_T(\cdot)\hspace{-0.2cm} $& $g_{1n}$&$g_{2n}$&$g_{3n}$&$g_{4n}$&$g_{5n}$&$g_{6n}$&$g_{7n}$&$g_{8n}$\\
 \hline
400/300&\textbf{.025}&.045&.040&.040&.040&.040&.040&.040&.040\\
&(.015)&(.015)&(.015)&(.015)&(.015)&(.015)&(.015)&(.015)&(.015)\\
\hline
600/400&\textbf{.020}&.035&.035&.035&.035&.035&.035&.035&.035\\
&(.015)&(.015)&(.015)&(.015)&(.015)&(.015)&(.015)&(.015)&(.015)\\
\hline
800/600&\textbf{.020}&.035&.035&.030&.030&.030&.030&.032&.030\\
&(.007)&(.015)&(.015)&(.015)&(.015)&(.015)&(.015)&(.011)&(.015)\\
 \hline
 \end{tabular}
 \end{center}}
  \vspace{-0.5cm}
 \caption{Median and MAD of the misclassification error rate over $500$ replicates for  
$\mathbb{R}^{150}$ with fixed number of neighbors.}\label{tabla12}
 \end{table}

In Table \ref{tabla21} we report the mean of the misclassification error rate and standard 
deviation for case \ref{2}, with the original aggregated classifier and the two more flexible 
versions: $\alpha=1/8$ and  $\alpha=1/4$. In this table we compare the performance of our rules with 
the (optimal) cross validated nearest neighbor classifier computed with $k$ and also with $n$. In 
Table \ref{tabla22} we report the median and MAD of the misclassification error rate for this 
case. 

\begin{table}[b!]
 \footnotesize{
 \begin{center}
 \begin{tabular}{|c|c|c|c|c|c|}
 \hline
 n/k  & $g_T(\cdot)$ & $g_T(\cdot, 1/8)$ &  $g_T(\cdot, 1/4)$ & $gcv_{n}$ & $gcv_{k}$\\
 \hline
 400/300&\textbf{.029}&.038&.046&.040&.044\\
&(.016)&(.019)&(.021)&(.017)&(.018)\\
 600/400&\textbf{.029}&.039&.047&.037&.043\\
&(.016)&(.019)&(.022)&(.016)&(.018)\\
 800/600&\textbf{.027}&.036&.046&.033&.036\\
&(.014)&(.018)&(.020)&(.014)&(.015)\\
\hline
 \end{tabular}
 \end{center}}
  \vspace{-0.5cm}
 \caption{Mean and standard deviation of the misclassification error rate over $500$ replicates for 
$\mathbb{R}^{150}$ with the number of neighbors chosen at random.}\label{tabla21}
 \end{table}
  \begin{table}[b!]
 \footnotesize{
 \begin{center}
 \begin{tabular}{|c|c|c|c|c|c|}
 \hline
 n/k  & $g_T(\cdot)$ & $g_T(\cdot, 1/8)$ &  $g_T(\cdot, 1/4)$ & $gcv_{n}$ & $gcv_{k}$\\
 \hline
 400/300 &\textbf{.025}&.035&.045&.040&.042\\
&(.015)&(.015)&(.022)&(.015)&(.019)\\
 600/400&\textbf{.028}&.035&.045&.035&.040\\
&(.019)&(.015)&(.022)&(.015)&(.015)\\    
 800/600&\textbf{.025}&.035&.045&.035&.035\\
&(.015)&(.015)&(.022)&(.015)&(.015)\\
 \hline
 \end{tabular}
 \end{center}}
  \vspace{-0.5cm}
 \caption{Median and MAD of the misclassification error rate over $500$ replicates for 
$\mathbb{R}^{150}$ with the number of neighbors chosen at random.}\label{tabla22}
 \end{table}

\begin{itemize}
\item \underline{Heterogeneous case}: $M=5$ classifiers: 3 $k$-nearest neighbor rules with fixed values of $k$, the Fisher and the random forest classifiers.
\end{itemize}
Here we take $3,5,7$ nearest neighbors (denoted by $g_{mn}$ for $m=1,2,3$), the Fisher classifier (denoted by $g_{F}$) and the random forest classifier (denoted by $g_{RF}$). In Table \ref{tabla31} we report the averaged misclassification error rates and standard deviation and in Table \ref{tabla32} we report the median and MAD for this case.  
 \begin{table}[t!]
\footnotesize{ 
 \begin{center}
  \begin{tabular}{|c|c|c|c|c|c|c|}
 \hline
 n/k   &$g_T(\cdot)$& $g_{1n}$&$g_{2n}$&$g_{3n}$&$g_{F}$&$g_{RF}$\\
 \hline
400/300&.012&.049&.043&.041&.020&\textbf{.004}\\
&(.011)&(.016)&(.017)&(.017)&(.011)&(.004)\\
\hline  
 600/400&.008&.047&.040&.037&.012&\textbf{.001}\\
&(.007)&(.015)&(.015)&(.015)&(.008)&(.002)\\
 \hline
 800/600&.007&.043&.036&.034&.009&\textbf{.000}\\
&(.007)&(.015)&(.015)&(.014)&(.007)&(.002)\\
\hline
\end{tabular}
\end{center}}
 \vspace{-0.5cm}
\caption{Mean and standard deviation of the misclassification error rate over $500$ replicates for 
 $\mathbb{R}^{150}$ with fixed number of neighbors, Fisher classifier and random 
forest.}\label{tabla31}
\end{table}

\begin{table}[h!]
\footnotesize{ 
 \begin{center}
  \begin{tabular}{|c|c|c|c|c|c|c|}
 \hline
 n/k    &$g_T(\cdot)$& $g_{1n}$&$g_{2n}$&$g_{3n}$&$g_{F}$&$g_{RF}$\\
 \hline
400/300&.010&.050&.040&.040&.020&\textbf{.000}\\
&(.007)&(.015)&(.015)&(.015)&(.015)&(.000)\\
\hline
600/400&.005&.045&.040&.035&.010&\textbf{.000}\\
&(.007)&(.015)&(.015)&(.015)&(.007)&(.000)\\
\hline
800/600&.005&.040&.035&.035&.010&\textbf{.000}\\
&(.007)&(.015)&(.015)&(.015)&(.007)&(.000)\\
 \hline
 \end{tabular}
 \end{center}}
  \vspace{-0.5cm}
 \caption{Median and MAD of the misclassification error rate  over $500$ replicates for  
$\mathbb{R}^{150}$ with fixed number of neighbors, Fisher classifier and random 
forest.}\label{tabla32}
 \end{table}

\begin{figure}[b!]
\begin{center}
\includegraphics[width=0.45\textwidth]{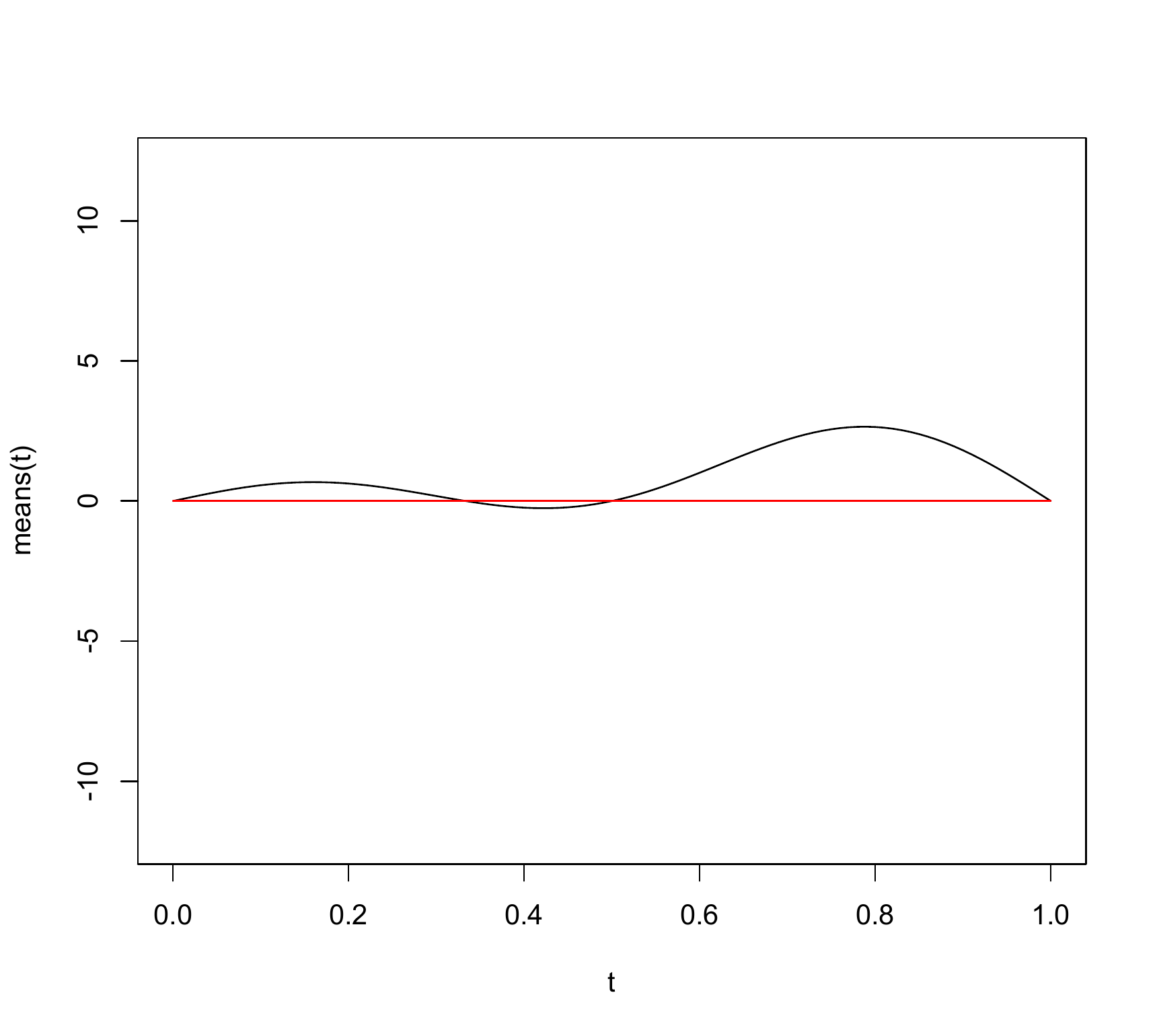} 
\includegraphics[width=0.45\textwidth]{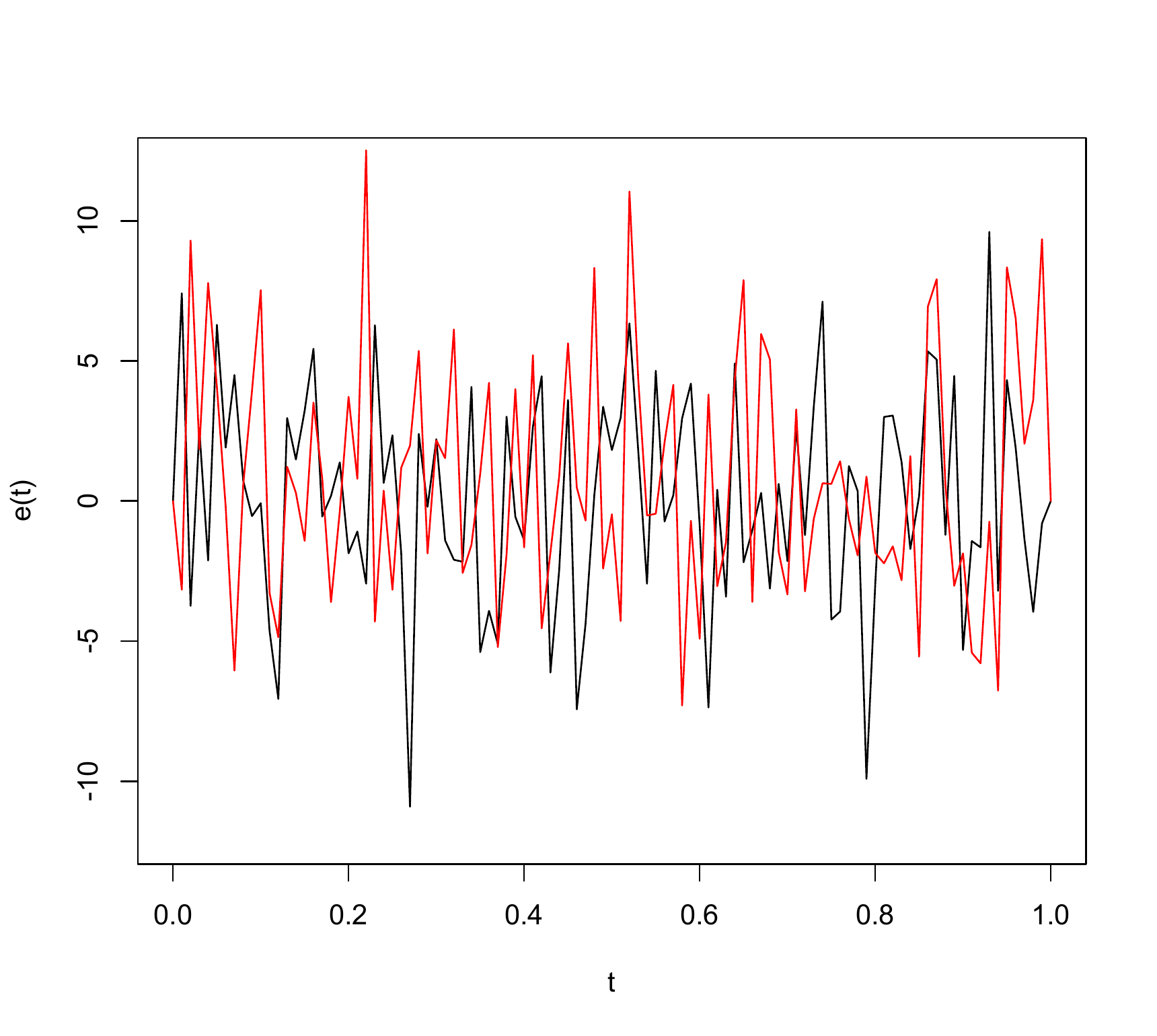}  
  \vspace{-0.5cm}
\caption{Mean curve (Left) and Error curve (Right) of the two populations of Model II.}
\label{figmodiimean}
\end{center}
\end{figure}
\subsection*{Functional data setting}

In this setting we show the performance of our method by analyzing the following two models 
considered in Delaigle and Hall \cite{hall_2012}:
\begin{itemize}
\item Model I: We generate two samples of size $n/2$ from different populations following the
model
\begin{equation*}\label{modelo-hall}	
X_{pi}(t) = \sum_{j=1}^6 \mu_{p,j} \phi_j(t) + e_{pi}(t), \hspace{1cm} p = 1,2,\hspace{0.5cm}  i =
1,\ldots, n/2,
\end{equation*}
where $\phi_j(t) = \sqrt{2} \sin(\pi j t)$, $\mu_{1,j}$ and $\mu_{2,j}$ are, respectively, the j-th coordinate of the mean vectors $\mu_{1} = (0,-0.5, 1, -0.5, 1, -0.5)$, and $\mu_{2} = (0, -0.75, 0.75, -0.15, 1.4, 0.1)$ while the errors are given by
\[
e_{pi}(t) = \sum_{j=1}^{40} \sqrt{\theta_j} Z_{pj} \phi_j(t), \hspace{1cm} p = 1,2,
\]
with $Z_{pj}\sim \mathcal{N}(0,1)$ and $\theta_j = 1/j^2$.

\item Model II: We generate two samples of size $n/2$ from different populations following the model
\begin{equation*}\label{modelo-hall2}	
X_{pi}(t) = \sum_{j=1}^3 \mu_{p,j} \phi_j(t) + e_{pi}(t), \hspace{1cm} p = 1,2,\hspace{0.5cm}  i =
1,\ldots, n/2,
\end{equation*}
where $\mu_{1} = 0.75\cdot(1, -1, 1)$ and $\mu_{2,j}$ the j-th coordinate of $\mu_{2} \equiv 0$,
$\theta_j = 1/j^2$ and the errors are given by
\[
e_{pi}(t) = \sum_{j=1}^{40} \sqrt{\theta_j} Z_{pj} \phi_j(t), \hspace{1cm} p = 1,2,
\]
with $Z_{pj}\sim \mathcal{N}(0,1)$ and $\theta_j = \exp\{- (2.1 - (j-1)/20)^2 \}$. 

This second model looks more challenging since although the means of the two populations are quite
different, the error process is very wiggly, concentrated in high frequencies (as shown in Figure
\ref{figmodiimean} left and right panel, respectively). So in this case, in order to  apply our
classification method, we have first performed the Nadaraya-Watson kernel smoother
(taking a normal kernel) to the training sample with different values of the bandwidths 
for each of the two populations. The values for the bandwidths were chosen via cross-validation 
with our classifier, varying the bandwidths between $.1$ and $.7$ (in intervals of length $.05$).
The optimal values, over 200 replicates, were $h_1=.15$ for the first population
(with mean $\mu_1$) and $h_2=.7$ for the second one. Finally, we apply the classification
method to the raw (non-smoothed) curves of the testing sample.
\end{itemize}
\begin{table}[!t]
\begin{center}
\footnotesize{
\begin{tabular}{|c|c|c|c|c|c|c|c|c|c|}
\hline
 \hspace{-0.1cm}Model \hspace{-0.2cm}   & $g_T(\cdot)$ & $\hspace{-0.15cm}g_T(\cdot, 1/5)\hspace{-0.15cm}$ &  $\hspace{-0.15cm}g_T(\cdot,  2/5)\hspace{-0.15cm}$ & $\hspace{-0.15cm}g_T(\cdot,3/5)\hspace{-0.15cm}$ &$\hspace{-0.15cm}g_{1n}\hspace{-0.15cm}$ & $\hspace{-0.15cm}g_{2n}\hspace{-0.15cm}$ & $\hspace{-0.15cm}g_{3n}\hspace{-0.15cm}$  & $\hspace{-0.15cm}g_{4n}\hspace{-0.15cm}$  & $\hspace{-0.15cm}g_{5n}\hspace{-0.15cm}$\\
\hline
\multirow{2}{*}{I}   & .013  & .005 & .004 &.005  &.017  &.007  &.004  &.003  &\textbf{.002}\\
                     &(.011) &(.005)&(.005)&(.006)&(.011)&(.007)&(.004)&(.004)&(.003)\\
 \hline                             
  \multirow{2}{*}{II}& .110  & .074 &.068  &.069  &.124  &.083 &.070  &.066  &\textbf{.064}\\
                     &(.029)&(.019)&(.018)&(.018)&(.030)&(.020)&(.018)&(.016)&(.017)\\
  \hline 
 \end{tabular}}
 \end{center}
  \vspace{-0.5cm}
 \caption{Mean and standard deviation of the misclassification error rate over $200$ replicates for 
models I and II.}\label{tabla51}
 \end{table}
\begin{table}[!b]
\begin{center}
\footnotesize{
\begin{tabular}{|c|c|c|c|c|c|c|c|c|c|}
\hline
 \hspace{-0.1cm}Model \hspace{-0.2cm}   & $g_T(\cdot)$ & $\hspace{-0.15cm}g_T(\cdot,1/5)\hspace{-0.15cm}$ &  $\hspace{-0.15cm}g_T(\cdot,  2/5)\hspace{-0.15cm}$ & $\hspace{-0.15cm}g_T(\cdot,3/5)\hspace{-0.15cm}$ &$\hspace{-0.15cm}g_{1n}\hspace{-0.15cm}$ & $\hspace{-0.15cm}g_{2n}\hspace{-0.15cm}$ & $\hspace{-0.15cm}g_{3n}\hspace{-0.15cm}$  & $\hspace{-0.15cm}g_{4n}\hspace{-0.15cm}$  & $\hspace{-0.15cm}g_{5n}\hspace{-0.15cm}$\\
\hline
\multirow{2}{*}{I}   & .010  & .005 & .004 &.005  &.015  &.005  &\textbf{.000} &\textbf{.000}  &\textbf{.000}\\
                     &(.007) &(.007)&(.007)&(.007)&(.007)&(.007)&(.000)&(.000)&(.000)\\
 \hline                             
  \multirow{2}{*}{II}& .105  & .070 &\textbf{.065}  &.070  &.120  &.080 &.070  &\textbf{.065}  &\textbf{.065}\\
                     &(.030)&(.015)&(.015)&(.015)&(.030)&(.022)&(.022)&(.015)&(.015)\\
  \hline 
 \end{tabular}}
 \end{center}
  \vspace{-0.5cm}
 \caption{Median and MAD of the misclassification error rate  over $200$ replicates for models I 
and II.}\label{tabla52}
 \end{table}
In Table \ref{tabla51} we report the averaged misclassification error rate and the standard deviation over $200$ replications for models  I and II, taking $n=90$, $k=60$, $l=30$, and $\alpha = 0,1,2,3$. In the whole training sample (of $n$
functions) the $n/2$ labels for every population were chosen at random. The test sample consist of $200$ data, taking $100$ of every population. Here, $g_{mn} = (2m-1)$-nearest neighbor rule for $m=1,\dots,5$. In Table \ref{tabla52} we report the median of the misclassification error rate and the MAD. For Model I we get a better performance than the PLS-Centroid Classifier proposed by Delaigle and Hall \cite{hall_2012}. For model II PLS-Centroid Classifier clearly outperforms our classifier although we get a quite small missclassification error, just using a combination of five nearest neighbor estimates. 

\section{A real data example: Analysis of spectrograms} \label{realdata}

\begin{figure}[h]
\begin{center}
\includegraphics[width=6cm]{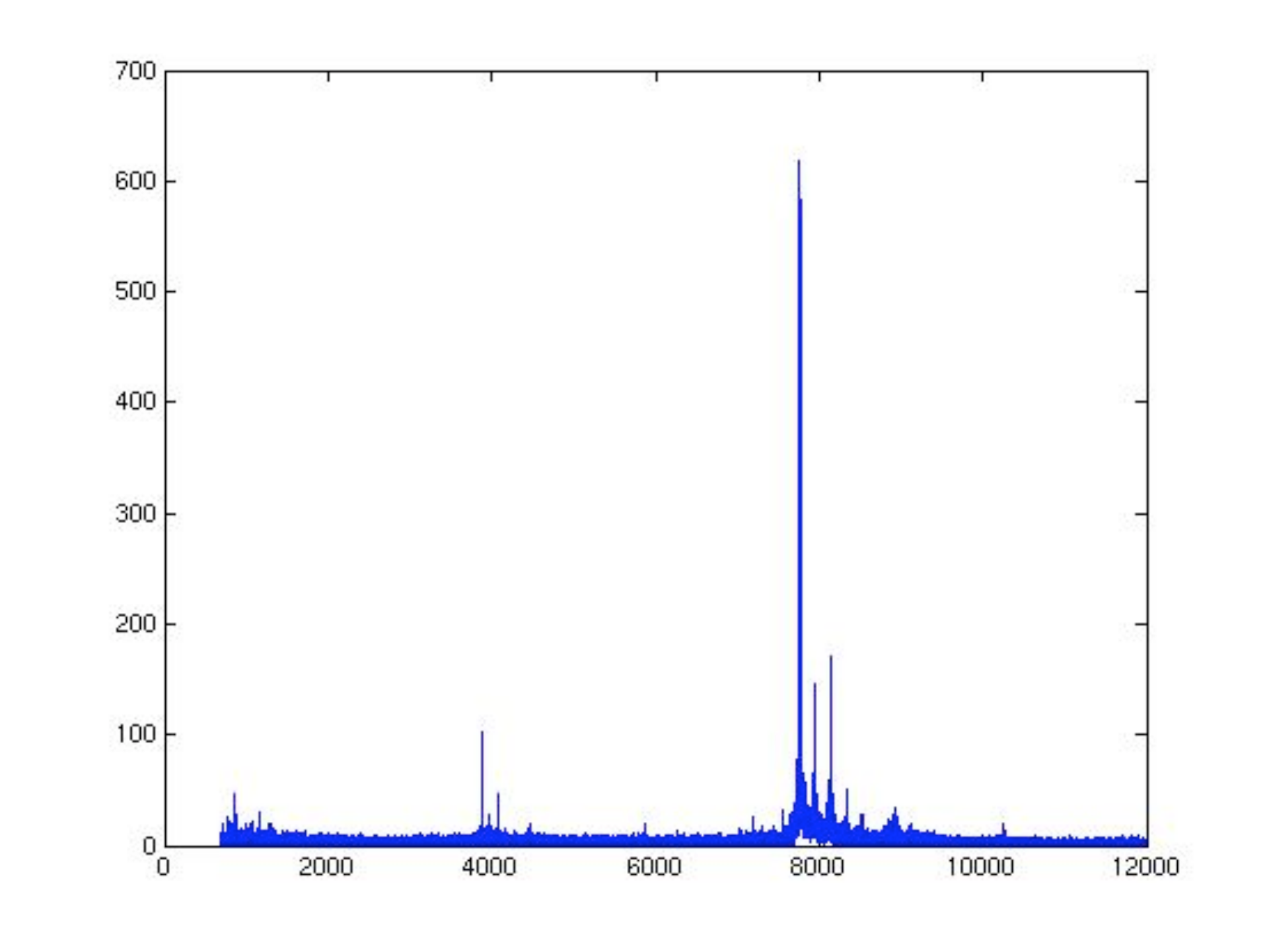}
\includegraphics[width=6cm]{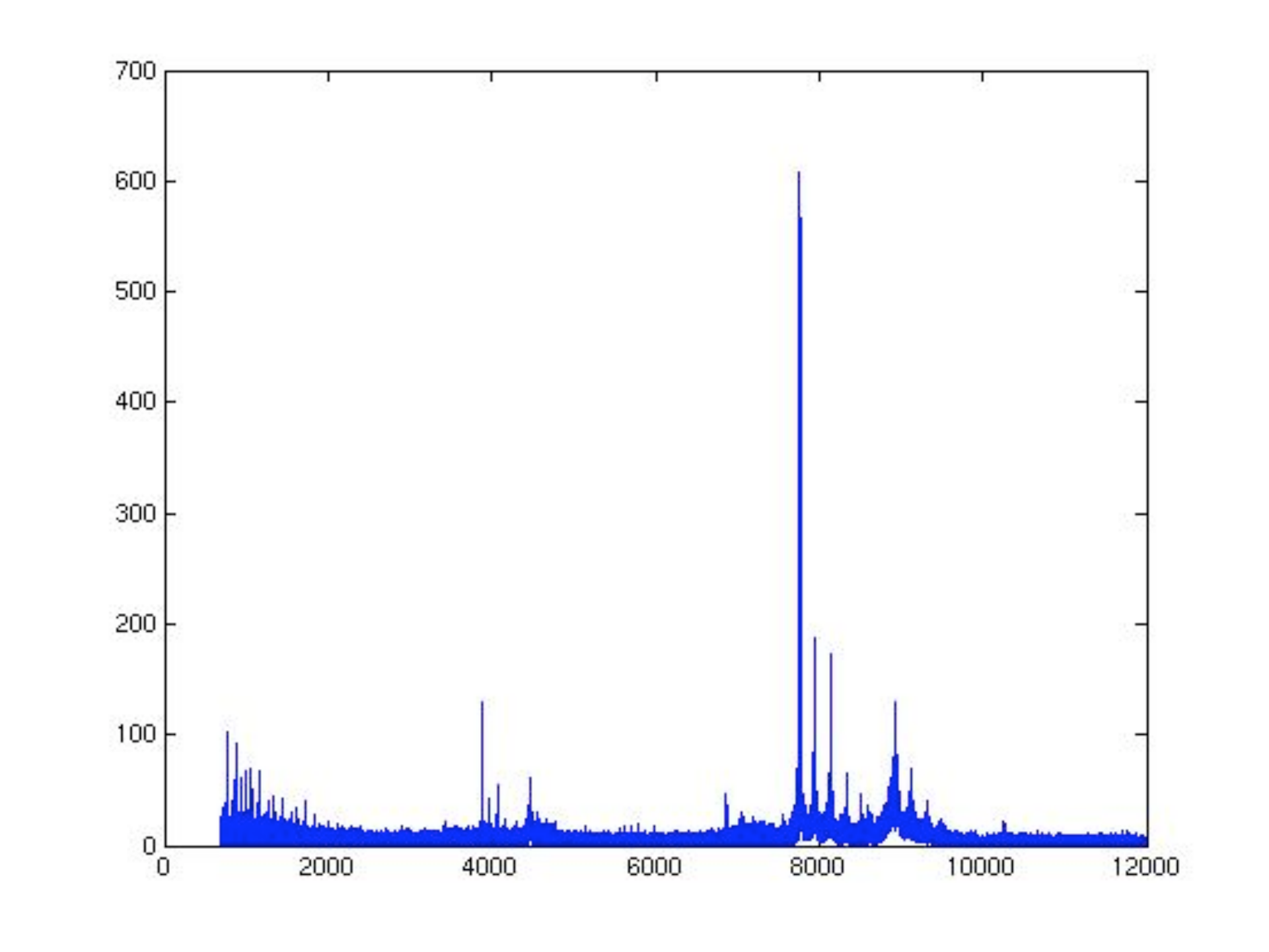}
\end{center}
\vspace{-15pt} \caption{Spectrogram of a healthy (left panel) and a ovarian cancer suffering woman (right panel).} \label{paisanas}
\end{figure}

The data to be analyzed in this section consists in the mass spectra from blood samples of 216 women of which, 121 suffer from an ovarian cancer condition and the remaining  95 are healthy women which were taken as control group. We refer to  \cite{ban03} for a previous analysis of these data with a detailed discussion of their medical aspects, see also \cite{cues06} for further statistical analysis of these data.

A spectrogram is a curve showing the number of molecules (or
fragments) found for every mass/charge ratio and, the idea behind spectrograms, is to control the amount of proteins  produced in cells since, when cancer starts to grow, its cells produce a different kind of proteins than those produced by healthy cells. Moreover, the amount of common produced proteins may be different. Proteomics, broadly speaking, consists of a family of procedures
allowing researchers to analyze proteins. In particular, here we
are interested in some techniques which allow to separate mixtures
of complex molecules according to the rate mass/charge (observe that,  molecules with the same mass/charge ratio are indistinguishable with a spectrogram). 

We have processed the data as follows: we have restricted ourselves to the interval 
mass charge (horizontal axis) $[7000,9500]$. Then, in order to have all the spectra defined in a common equi-spaced grid, we have smoothed them via a Nadaraya-Watson smoother. Finally, every function has been divided by its maximum, in order to have all the values scaled in the common interval $[0,1]$. Observe that our interest is to find the location of maxima amount of molecules more than the corresponding heights. 

To build the classifier introduced in (\ref{clasifgen}) we have taken $5$ 
nearest neighbor classifiers, with $k=3,5,7,9$ neighbors. We have implemented the cross validation method in a grid for $(\alpha,l)$, with $\alpha$ taking the values $0,1/5,2/5$ and $l$  taking $60$ values $l=20,21,\dots,80$. The minimum of the misclassification error was attained for $\alpha=2$ and $l=31,\dots,36$ in whose case the accuracy obtained was 95\%.

\section{Concluding remarks}
\begin{itemize}
\item  We introduce a new nonlinear aggregating method for supervised
classification in a general setup built up from a family of
 classifiers $g_{1k},\dots,g_{Mk}$.  It combines the decision of the M experts according to a ``coincidence opinion" with respect to the new data we want to classify.
 \item The new method, besides being easy to implement, is particularly well designed for high dimensional and functional data. The 
 method is not local, and the use of the inverse functions prevent from the curse of dimensionality that suffers all local methods.
 \item We obtain consistency and rates of convergence under very mild conditions on a general metric space setup.
 \item An optimality result is obtained in the sense that the nonlinear  aggregation rule behaves asymptotically as  well as
the best one among  the $M$ classifiers  (experts) $g_{1k},\dots,g_{Mk}$. 
\item A small simulation study  confirms
the asymptotic results for moderate sample sizes. In particular it is very well behaved for 
high--dimensional and functional data.
\item In a well known spectrogram curves dataset, we obtain a very good 
performance, classifying 95\%, very close to the best known results for these data.
\item Although we have implemented cross validation to choose the parameters $(\alpha,l)$ in Section \ref{realdata}, conditions
for the validity of this procedure remains as an open problem.
\end{itemize}
 
\section{Appendix: Proof of results}
To prove Theorem \ref{consistencia} we will need the following Lemma.
\begin{lem}\label{lema}
Let $f(x)$ be a classifier built up from the training sample $\mathcal{D}_k$ such that
$\mathbb P_{\mathcal{D}_k}(f(X)\neq g^*(X)) \rightarrow 0$ when
$k\rightarrow \infty$. Then, $\mathbb P_{\mathcal{D}_k}(f(X)\neq Y)-L^*\rightarrow 0$.
\end{lem}
\begin{proof} [Proof of Lemma \ref{lema}]
First we write,
\begin{align}\label{error}
\mathbb P_{\mathcal{D}_k}\big(f(X)\neq Y\big)- L^* &= \mathbb P_{\mathcal{D}_k}\big(f(X)\neq
Y\big)-P\big(g^*(X)\neq Y\big) \nonumber\\ &= \mathbb P_{\mathcal{D}_k}\big(f(X)\neq Y,Y=g^*(X)\big)
\nonumber\\ &\hspace{0.3cm}+ \mathbb P_{\mathcal{D}_k}\big(f(X)\neq Y, Y\neq g^*(X)\big)
-P\big(g^*(X)\neq Y\big)\nonumber \\ &= \mathbb P_{\mathcal{D}_k}\big(f(X)\neq g^*(X)\big)\\
&\hspace{0.3cm} + \mathbb P_{\mathcal{D}_k}\big(f(X)\neq Y, Y\neq g^*(X)\big)-P\big(g^*(X)\neq
Y\big) \nonumber \\ &= \mathbb P_{\mathcal{D}_k}\big(f(X)\neq g^*(X)\big) -
\mathbb P_{\mathcal{D}_k}\big(g^*(X)\neq Y, f(X)= Y\big), \nonumber
\end{align}
where in the last equality we have used that 
\[
P\big(g^*(X)\neq Y\big)= \mathbb P_{\mathcal{D}_k}\big(g^*(X)\neq Y, f(X)\neq Y\big) +
\mathbb P_{\mathcal{D}_k}\big(g^*(X)\neq Y, f(X)= Y\big),
\]
implies 
\[
 \mathbb P_{\mathcal{D}_k}\big(g^*(X)\neq Y,f(X)= Y\big) = \mathbb P_{\mathcal{D}_k}\big(g^*(X)\neq
Y, f(X)\neq Y\big) - P\big(g^*(X)\neq Y\big).
\]
Therefore, replacing in (\ref{error}) we get that
\begin{align} \label{boundl3}
\mathbb P_{\mathcal{D}_k}\big(f(X)\neq Y\big)- L^* &=  \mathbb P_{\mathcal{D}_k}\big(f(X)\neq
g^*(X)\big) -
\mathbb P_{\mathcal{D}_k}\big(g^*(X)\neq Y, f(X)= Y\big)\nonumber \\ &\le \mathbb
P_{\mathcal{D}_k}\big(f(X)\neq g^*(X)\big),
\end{align}
which by hypothesis converges to zero as $k \to \infty$ and the Lemma is proved.
\end{proof}
\begin{proof} [Proof of Theorem \ref{consistencia}]
We will prove part b) of the Theorem since part a) is a direct consequence of it. By (\ref{boundl3}), 
it suffices to prove that, for $k$ large enough: 
$$\mathbb P_{\mathcal{D}_k}(g_T(X,\alpha)\neq g^*(X))=\mathcal{O}\Big(\max\big\{\exp(-C(n-k)),\bm{\beta_{Rk}}\big\}\Big).$$ 
We first split $\mathbb P_{\mathcal{D}_k}(g_T(X,\alpha)\neq g^*(X))$ into two terms,
\begin{align*} %\label{eq1}
\mathbb P_{\mathcal{D}_k}(g_T(X,\alpha)\neq g^*(X)) &= \mathbb P_{\mathcal{D}_k}(g_T(X,\alpha)\neq g^*(X),g^*(X)=1) \nonumber \\ &\hspace{0.3cm}+\mathbb P_{\mathcal{D}_k}(g_T(X,\alpha)\neq g^*(X),g^*(X)=0) \doteq I +II.
\end{align*}
Then we will prove that, for $k$ large enough,
$$I=\mathcal{O}\Big(\max\big\{\exp(-C_1(n-k)),\bm{\beta_{Rk}}\big\}\Big),$$ 
for some arbitrary constant $C_1$. The proof that 
$$II=\mathcal{O}\Big(\max\big\{\exp(-C_2(n-k)),\bm{\beta_{Rk}}\big\}\Big),$$  
for some arbitrary 
constant $C_2$ is completely analogous and we omit it.
Finally, taking $C=\min\{C_1,C_2\}$, the proof will be completed.
In order to deal with term $I$, let us define the vectors 
\begin{align*}
\mathbf{g_{Rk}}(X) =\  & \big(g_{1k}(X),\dots,g_{Rk}(X)\big)\in \{0,1\}^R,\\ 
\nu(X)= \ & \Big(1,\dots,1,g_{(R+1)k}(X),\dots,g_{Mk}(X)\Big)\in \{0,1\}^M.
\end{align*}
Then, 
\begin{align*}
I &= \mathbb P_{\mathcal{D}_k}(g_T(X,\alpha)\neq g^*(X),g^*(X)=1)  \nonumber\\
 &\le \mathbb P_{\mathcal{D}_k}(g_T(X,\alpha)\neq g^*(X),g^*(X)=1, \mathbf{g_{Rk}}(X)=\mathbf{1})
\nonumber \\ &\hspace{0.5cm}+\sum_{m=1}^R \mathbb P_{\mathcal{D}_k}(g_T(X,\alpha)\neq
g^*(X),g^*(X)=1,g_{mk}(X)=0)\nonumber \\
&\le \mathbb P_{\mathcal{D}_k}(g_T(X,\alpha)\neq g^*(X),g^*(X)=1,
\mathbf{g_{Rk}}(X)=\mathbf{1})\nonumber \\ 
&\hspace{0.5cm}+\sum_{m=1}^R \mathbb P_{\mathcal{D}_k}(g^*(X)\neq g_{mk}(X)) \\ &\le \mathbb
P_{\mathcal{D}_k}\Big(T_n(\mathbf{g_k}(X), \alpha)\leq
1/2\big|g^*(X)=1,\mathbf{g_{Rk}}(X)=\mathbf{1}\Big) 
 \\ &\hspace{0.5cm}+\sum_{m=1}^R \mathbb P_{\mathcal{D}_k}(g^*(X)\neq g_{mk}(X))
\\ &\doteq I_A + I_B.\nonumber 
\end{align*} 
Observe that, conditioning to $\mathbf{g_{Rk}}(X)=\mathbf{1}$ and defining 
$$
Z_{j} \doteq \mathbb{I}_{\left\{\frac{1}{M}\sum_{m=1}^M \mathbb{I}_{\{g_{mk}(X_j)= \nu(m)\}} \ge
1-\alpha\right\}},
$$
we can rewrite $T_n(\mathbf{g_{k}}(X), \alpha) $ as
\[
T_n(\mathbf{g_{k}}(X), \alpha) = \frac{\sum_{j=k+1}^n  Z_j Y_j}{\sum_{i=k+1}^n Z_i}.
\]
Therefore, 
\begin{align}\label{proba}
I_A&= \mathbb P_{\mathcal{D}_k}\left(\frac{\frac{1}{n-k}\sum_{j=k+1}^n  Z_j
Y_j}{\frac{1}{n-k}\sum_{i=k+1}^n Z_i} \le \frac{1}{2} \Big| g^*(X)=1,
\mathbf{g_{Rk}}(X)=\mathbf{1}\right)\nonumber \\ &= \mathbb
P_{\mathcal{D}_k}\left(\frac{1}{n-k}\sum_{j=k+1}^n  Z_j (Y_j-1/2) \le 0  \Big| g^*(X)=1,
\mathbf{g_{Rk}}(X)=\mathbf{1}\right). %\\&=\exp(-C_1(n-k)). \nonumber
\end{align}
In order to use a concentration inequality to bound this probability, we need to compute the
expectation of $Z_j (Y_j-1/2) =
Z_j Y_j - Z_j/2 $. To do this, observe that 
\[
E(Z_j Y_j)=\mathbb P_{\mathcal{D}_k}\left(\frac{1}{M}\sum_{m=1}^M \mathbb{I}_{\{g_{mk}(X)= \nu(m) 
\}} \ge1-\alpha,Y=1\right),
\]
and 
\begin{equation*}\label{esp1}
E(Z_j) = \mathbb P_{\mathcal{D}_k}\left(\frac{1}{M}\sum_{m=1}^M \mathbb{I}_{\{g_{mk}(X)=\nu(m)\}} 
\ge1-\alpha \right).
\end{equation*}
Since
\begin{equation*} \label{inc1}
\left\{ \mathbf{g_{Rk}}(X)=\mathbf{1} \right\} \subset \left\{ \frac{1}{M}\sum_{m=1}^M
\mathbb{I}_{\{g_{mk}(X)= \nu(m)\}} \ge 1-\alpha \right\}\doteq A_\alpha,
\end{equation*}
we have,
\begin{align}\label{esp2}
E(Z_j Y_j)-E(Z_j)/2 & = \mathbb P_{\mathcal{D}_k}(V=1|A_\alpha)\mathbb 
P_{\mathcal{D}_k}(A_\alpha)-\mathbb P_{\mathcal{D}_k}(A_\alpha)/2 \nonumber\\
& = \mathbb P_{\mathcal{D}_k}(A_\alpha)\Big[\mathbb P_{\mathcal{D}_k}(Y=1|A_\alpha)-1/2\Big]\\
& \geq \mathbb{P}_{\mathcal{D}_k}\big(\mathbf{g_{Rk}}(X)=1\big)\Big[\mathbb
P_{\mathcal{D}_k}(Y=1|A_\alpha)-1/2\Big].\nonumber
\end{align}
Now, since for $m=1,\ldots,R$, $g_{mk} \to g^*$ in probability as $k \to \infty$, 
\begin{equation} \label{lim1}
\mathbb P_{\mathcal{D}_k}\left(\mathbf{g_{Rk}}(X)=\mathbf{1}\right) \to \mathbb{P}(g^*(X)=1)\doteq 
p^*>0.
\end{equation}
On the other hand, we have that, for $k$ large enough,
$\mathbb P_{\mathcal{D}_k}(Y=1|A_\alpha)> 1/2$. Indeed, for $m=1,\dots,R$, let us consider the
events $B_{mk}=\{g_{mk}(X)=g^*(X)\}$ which, by hypothesis, for $k$ large enough verify
$$
\mathbb P \big(\cap_{m=1}^R B_{mk}\big)>1-\varepsilon,
$$
for all $\varepsilon>0$. In particular, we can take $\varepsilon>0$ such that $\mathbb
P(Y=1|g^*(X)=1)(1-\varepsilon)>1/2$. This implies that 
\begin{align} \label{eq}
\mathbb P_{\mathcal{D}_k}(Y=1|A_\alpha)= \ & \frac{\mathbb 
P_{\mathcal{D}_k}(Y=1,A_\alpha,\cap_{m=1}^R B_{mk})}{\mathbb P_{\mathcal{D}_k}(A_\alpha)} 
\nonumber\\
&\quad + \frac{\mathbb P_{\mathcal{D}_k}(Y=1,A_\alpha,(\cap_{m=1}^R B_{mk})^c)}{\mathbb 
P_{\mathcal{D}_k}(A_\alpha)}\nonumber\\
\geq \ &\frac{\mathbb P_{\mathcal{D}_k}(Y=1,A_\alpha,\cap_{m=1}^R B_{mk})}{\mathbb
P_{\mathcal{D}_k}(A_\alpha)}\\
> \ &\frac{\mathbb P_{\mathcal{D}_k}(Y=1,A_\alpha\big|\cap_{m=1}^RB_{mk})} {\mathbb
P_{\mathcal{D}_k}(A_\alpha)}(1-\varepsilon).\nonumber
\end{align}
Conditioning to $\cap_{m=1}^R B_{mk}$ the event $A_\alpha$ equals $C_\alpha$ given by
\begin{equation} \label{condaalpha}
\left\{R\mathbb{I}_{\{g^*(X)= 1\}} +\sum_{m=R+1}^M \mathbb{I}_{\{g_{mk}(X)= \nu(m)\}}\ge M(1-\alpha)
\right\}\doteq C_\alpha.
\end{equation}
However, $\alpha<1/2$ imply that $C_\alpha=\{g^*(X)=1\}$. Indeed, from the inequality  $R\geq 
M(1-\alpha)$, it is clear that $\{g^*(X)=1\}\subset C_\alpha$. 
On the other hand, $R\geq M(1-\alpha)>M/2$ and $\alpha<1/2$ imply that $M-R<M/2<M(1-\alpha)$, 
and so the sum in the second term of (\ref{condaalpha}) is at most $M-R$ and consequently, 
$\{g^*(X)=1\}^c\subset C_\alpha^c$. Then, combining this fact with (\ref{eq}) we have that, for $k$ 
large enough

\begin{align}\label{otraeq}
\mathbb P_{\mathcal{D}_k}(Y=1|A_\alpha)&\geq \  \frac{\mathbb
P_{\mathcal{D}_k}(Y=1,g^*(X)=1\big|\cap_{m=1}^RB_{mk})}{\mathbb
P_{\mathcal{D}_k}(g^*(X)=1)}(1-\varepsilon) \nonumber \\
&=\ \frac{\mathbb P_{\mathcal{D}_k}(Y=1,g^*(X)=1)}{\mathbb 
P_{\mathcal{D}_k}(g^*(X)=1)}(1-\varepsilon)\\
&=\ \mathbb P(V=1|g^*(X)=1)(1-\varepsilon) \nonumber \\& >1/2. \nonumber
\end{align}
Therefore, from (\ref{lim1}) and (\ref{otraeq}) in (\ref{esp2}) we get 
\[
 E(Z_j Y_j)-E(Z_j)/2  > c >0.
\]
Going back to (\ref{proba}), conditioning to $\nu(X)$ and using the Hoeffding inequality 
for $|Z_j (Y_j-1/2)| \le 1/2$, for $k$ large enough we have
\begin{small}
\begin{align*}
I_A %&\le  \mathbb P_{\mathcal{D}_k}\left(\frac{1}{n-k}\sum_{j=k+1}^n  \big(Z_j (Y_j-1/2) - E(Z_j (Y_j-1/2)) \big)\le - a   \Big| g^*(X)=1, \mathbf{g_k}(X)=\mathbf{1}\right) \\ 
&= \mathbb P_{\mathcal{D}_k}\left(\frac{1}{n-k}\sum_{j=k+1}^n  -\big(Z_j (Y_j-1/2) - E(Z_j
(Y_j-1/2)) \big)\ge c \Big| g^*(X)=1, \mathbf{g_{Rk}}(X)=\mathbf{1}\right) \\ &\le 
\exp \left\{- C_1(n-k)\right\},
\end{align*}
\end{small}
with $C_1 = 2c^2$. On the other hand, by hypothesis we have  
\begin{equation*}
I_B=\sum_{m=1}^M \mathbb P_{\mathcal{D}_k}(g^*(X)\neq g_{mk}(X))=\mathcal{O}(\bm{\beta_{Rk}}),
\end{equation*}
which concludes the proof. 
\end{proof}

\begin{proof}[Proof of Theorem \ref{optimalidad}]
First we write,
\begin{align}\label{primera}
\mathbb{P}_{\mathcal{D}_k}\big(g_T(X) \neq Y \big) &=
\mathbb{P}_{\mathcal{D}_k}\big(T_n(\mathbf{g_k}(X)) > 1/2, Y = 0\big) \nonumber \\ &\hspace{1cm}+
\mathbb{P}_{\mathcal{D}_k}\big(T_n(\mathbf{g_k}(X)) \leq 1/2, Y = 1\big)\nonumber  \\&=\sum_{\nu \in \mathbb{C}} \mathbb{P}_{\mathcal{D}_k}\big(T_n(\mathbf{g_k}(X)) > 1/2, (X,Y) \in A_{\nu}^0 \big) \\ &\hspace{1cm}+ \sum_{\nu \in \mathbb{C}}\mathbb{P}_{\mathcal{D}_k}\big(T_n(\mathbf{g_k}(X)) \leq 1/2, (X,Y) \in A_{\nu}^1\big) \nonumber\\ &\doteq I + II.\nonumber
\end{align}

Let us take $\nu$ fixed. Observe that in this case, $T_n(\mathbf{g_k(X)})$ depends only on the subsample $\mathcal{E}_l$, therefore 
the events $(X_j,Y_j)\in A_\nu^{i}$ and $(X,Y)\in A_\nu^{i}$ are independent for all $i=0,1$, $j=k+1,\dots,n$. Then,
 \begin{align*}
I &= \sum_{\nu \in \mathbb{C}}\mathbb{P}_{\mathcal{D}_k}\big(T_n(\mathbf{g_k}(X),0) > 1/2,
(X,Y) \in A_{\nu}^0\big) \\ &= \sum_{\nu \in \mathbb{C}}
\mathbb{P}_{\mathcal{D}_k}\left(\frac{\sharp\{j:(X_j,Y_j) \in A_{\nu}^1\}}{l}
> \frac{\sharp\{j:(X_j,Y_j) \in A_{\nu}^0\}}{l}, (X,Y) \in
A_{\nu}^0 \right)\\
 &= \sum_{\nu \in \mathbb{C}}
\mathbb{P}_{\mathcal{D}_k}\left(\frac{\sharp\{j:(X_j,Y_j) \in A_{\nu}^1\}}{l}
> \frac{\sharp\{j:(X_j,Y_j) \in A_{\nu}^0\}}{l}\right)\mathbb{P}_{\mathcal{D}_k}\left( (X,Y) \in
A_{\nu}^0 \right)\\
&= \sum_{\nu \in \mathbb{C}} \mathbb{P}_{\mathcal{D}_k}\left(\frac{1}{l}\sum_{j=k+1}^n \mathbb{I}_{\{(X_j,Y_j) \in A_{\nu}^1\}} -\mathbb{I}_{\{(X_j,Y_j) \in A_{\nu}^0\}} > 0\right)\mathbb{P}_{\mathcal{D}_k}\left((X,Y) \in A_{\nu}^0 \right).
\end{align*}
Let us define 
\[
T_j^{\nu} \doteq \mathbb{I}_{\{(X_j,Y_j) \in A_{\nu}^1\}} -\mathbb{I}_{\{(X_j,Y_j) \in A_{\nu}^0\}},
\]
\[
p_\nu^i \doteq  \mathbb{P}_{\mathcal{D}_k}((X,Y)\in A_{\nu}^i), \hspace{0.5cm} i=0,1,
\]
\[
p_\nu\doteq E(T_j^{\nu}) =p_\nu^1-p_\nu^0,   K_1=2\min_{\{\nu:p_\nu\neq 0\}}p_\nu^2,
\text{ and }\sigma_\nu^2 \doteq E((T_j^{\nu})^2) =p_\nu^1+p_\nu^0.
\]
To bound term $I$, we will consider 3 cases, $p_\nu<0$, $p_\nu>0$ and $p_\nu=0$. Let us first assume that $p_\nu<0$. In this case,  using the Hoeffding inequality we have,

 \begin{equation}\label{Ia}
 \mathbb{P}_{\mathcal{D}_k}\left(\frac{1}{l} \sum_{j=k+1}^n (T_j^{\nu}-p_\nu)> -p_\nu\right)= \mathcal{O}_k\big(\exp(-K_1l)\big) .
\end{equation}
If $p_\nu>0$, using Hoeffding inequality again we get
\begin{align} \label{Ib}
\mathbb{P}_{\mathcal{D}_k}\left(\frac{1}{l}\sum_{j=k+1}^nT_j^{\nu}> 0\right)&= 1- \mathbb{P}_{\mathcal{D}_k}\left(\frac{1}{l}\sum_{j=k+1}^n T_j^{\nu}\leq 0\right)\nonumber\\&=1+\mathcal{O}_k\big(\exp(-K_1l)\big).
\end{align}

If $p_\nu=0$, since for all $\nu$ and $j$, $E(|T^\nu_j|^3)=1$, using the Berry-Esseen inequality we get
\begin{align} \label{Ib1}
\mathbb{P}_{\mathcal{D}_k}\left(\frac{1}{l}\sum_{j=k+1}^nT_j^{\nu}> 0\right)&=\Big[\frac{1}{2}+\mathcal{O}_k\big(l^{-1/2}\big)\Big]\mathbb{I}_{\{\sigma_\nu^2>0\}}.
\end{align}
Observe that, since $\mathbb{P}(Y=1)=\sum_\nu p_\nu^1$ and $\mathbb{P}(Y=0)=\sum_\nu p_\nu^0$, there exists $\nu$ such that $\sigma_\nu^2 > 0$. Then, from (\ref{Ia}), (\ref{Ib}) and (\ref{Ib1}) we get
\begin{small}
\begin{align} \label{eqI}
I=& \sum_{\nu \in \mathbb{C}} \mathbb{P}_{\mathcal{D}_k}\left(\frac{1}{l}\sum_{j=k+1}^n T_j^{\nu} > 0\right)p_\nu^0\nonumber\\=&\sum_{p_\nu<0}\mathbb{P}_{\mathcal{D}_k}\left(\frac{1}{l}\sum_{j=k+1}^n T_j^{\nu} > 0\right)p_\nu^0+\sum_{p_\nu>0}\mathbb{P}_{\mathcal{D}_k}\left(\frac{1}{l}\sum_{j=k+1}^n  T_j^{\nu}  > 0\right)p_\nu^0 \nonumber\\
&\hspace{1cm}+\sum_{p_\nu=0}\mathbb{P}_{\mathcal{D}_k}\left(\frac{1}{l}\sum_{j=k+1}^n  T_j^{\nu}  > 0\right)p_\nu^0 \\&=\mathcal{O}_k\big(\exp(-K_1l)\big)\mathbb{I}_{\{\exists \nu: p_\nu\neq 0\}} +\mathcal{O}_k\big(l^{-1/2}\big)\mathbb{I}_{\{\exists \nu: (p_\nu=0,\sigma_\nu^2 > 0)\}}  \nonumber\\ &\hspace{1cm}+\sum_{p_\nu>0} p_\nu^0 +\frac{1}{2}\sum_{p_\nu=0} p_\nu^0.\nonumber
\end{align}
\end{small}

Analogously, it is easy to prove that 
\begin{align} \label{eqII}
 II &= \mathcal{O}_k\big(\exp(-K_1l)\big)\mathbb{I}_{\{\exists \nu: p_\nu\neq 0\}} +\mathcal{O}_k\big(l^{-1/2}\big)\mathbb{I}_{\{\exists \nu: (p_\nu=0,\sigma^2_\nu > 0)\}} \nonumber \\&\hspace{1cm}+ \sum_{p_\nu<0} p_\nu^1+\frac{1}{2} \sum_{p_\nu=0}p_\nu^0,
\end{align}
where in the last term we have used that $p_\nu=0$ implies $p_\nu^0= p_\nu^1$. Therefore, with (\ref{eqI}) and (\ref{eqII}) in  (\ref{primera}) we get,
\begin{small}
\begin{align} \label{I+II}
\mathbb{P}_{\mathcal{D}_k}\big(g_T(X) \neq Y 
\big)&=\mathcal{O}_k\big(\exp(-K_1l)\big)\mathbb{I}_{\{\exists \nu: p_\nu\neq 0\}} 
+\mathcal{O}_k\big(l^{-1/2}\big)\mathbb{I}_{\{\exists \nu: (p_\nu=0,\sigma_\nu^2 > 0)\}}\nonumber
\\&\hspace{1cm}+\sum_{p_\nu>0}p_\nu^0+ \sum_{p_\nu<0}p_\nu^1+\sum_{p_\nu=0}p_\nu^0 \\ &= 
\mathcal{O}_k\big(\exp(-K_1l)\big)\mathbb{I}_{\{\exists \nu: p_\nu\neq 0\}} 
+\mathcal{O}_k\big(l^{-1/2}\big)\mathbb{I}_{\{\exists \nu: (p_\nu=0,\sigma_\nu^2 > 0)\}}\nonumber 
\\&\hspace{1cm}+ \sum_{\substack{\nu:\nu(m)=0\\ p_\nu>0}}p_\nu^0+\hspace{-0.1cm}
\sum_{\substack{\nu:\nu(m)=1\\ p_\nu>0}}p_\nu^0+ \hspace{-0.1cm}
\sum_{\substack{\nu:\nu(m)=0 \\ p_\nu<0}}p_\nu^1+\hspace{-0.1cm}
\sum_{\substack{\nu:\nu(m)=1 \\ p_\nu<0}}p_\nu^1+\hspace{-0.1cm}
\sum_{p_\nu=0}p_\nu^0.\nonumber
\end{align}
\end{small}
On the other hand, for each $m$ we have,
\begin{small}
\begin{align}\label{unosolo}
\noindent \mathbb{P}_{\mathcal{D}_k}\big(g_{mk}(X) \neq Y\big) 
&=
\mathbb{P}_{\mathcal{D}_k}\big(g_{mk}(X)=0, Y=1\big) + \mathbb{P}_{\mathcal{D}_k}\big(g_{mk}(X)=1,
Y=0\big) \nonumber \\ 
&=\mathbb{P}_{\mathcal{D}_k} \left( \bigcup_{\nu:\nu(m) = 0}  \hspace{-0.2cm} (X,Y) \in
A_{\nu}^1 \right) + \mathbb{P}_{\mathcal{D}_k}
\left(\bigcup_{\nu:\nu(m) = 1} \hspace{-0.2cm} (X,Y) \in A_{\nu}^0\right)  \nonumber\\ 
&= \sum_{\nu:\nu(m) = 0} p_\nu^1+ \sum_{\nu:\nu(m) = 1} p_\nu^0\\
&=  \sum_{\substack{\nu:\nu(m) = 0\\ p_\nu<0}} p_\nu^1+\hspace{-0.3cm}\sum_{\substack{\nu:\nu(m) = 0\\ p_\nu>0}} p_\nu^1+\hspace{-0.3cm}\sum_{\substack{\nu:\nu(m) = 1\\ p_\nu<0}} p_\nu^0+\hspace{-0.3cm}\sum_{\substack{\nu:\nu(m) = 1\\ p_\nu>0}} p_\nu^0+\hspace{-0.1cm}\sum_{p_\nu=0} p_\nu^0,\nonumber 
\end{align}
\end{small}
where in the last equality we used again that  $p_\nu=0$ implies $p_\nu^0= p_\nu^1$ to joint 
\[
\sum_{\substack{\nu:\nu(m) = 1\\ p_\nu=0}} p_\nu^0 + \sum_{\substack{\nu:\nu(m) = 0\\ p_\nu=0}} p_\nu^1 = \sum_{p_\nu=0} p_\nu^0. 
\]

Therefore, from (\ref{I+II}) and (\ref{unosolo}) we get
\begin{align*}
& \hspace{-1.5cm}\mathbb{P}_{\mathcal{D}_k}\big(g_T(X) \neq Y \big)-\mathbb{P}_{\mathcal{D}_k}\big(g_{mk}(X) \neq Y\big)\\
&=\mathcal{O}_k\big(\exp(-K_1l)\big)\mathbb{I}_{\{\exists \nu: p_\nu\neq 0\}} +\mathcal{O}_k\big(l^{-1/2}\big)\mathbb{I}_{\{\exists \nu: (p_\nu=0,\sigma_\nu^2 >0)\}}\\&\hspace{2cm}+\sum_{\substack{\nu:\nu(m)=0\\ p_\nu>0}} (p_\nu^0-p_\nu^1) +\sum_{\substack{\nu:\nu(m)=1\\ p_\nu<0}} (p_\nu^1-p_\nu^0) \\ &\leq \mathcal{O}_k\big(\exp(-K_1l)\big)\mathbb{I}_{\{\exists \nu: p_\nu\neq 0\}} +\mathcal{O}_k\big(l^{-1/2}\big)\mathbb{I}_{\{\exists \nu: (p_\nu=0,\sigma^2_\nu > 0)\}}.
\end{align*}
Observe that, if $p_\nu\neq 0$ for all $\nu$ we get
$\mathbb{P}_{\mathcal{D}_k}\big(g_T(X) \neq Y \big) -\mathbb{P}_{\mathcal{D}_k}\big(g_{mk}(X) \neq Y\big)\leq \mathcal{O}_k\big(\exp(-lK_1)\big)$.

\end{proof}

\section*{Acknowledgment} We would like to thank Gerard Biau and James Malley for helpful suggestions. We also thanks to the referees for their helpful suggestions which improved the presentation of this final version.


\section*{References}
\begin{thebibliography}{9}

\bibitem{baillo_2011} 
Ba\'illo, A., Cuevas, A., and Fraiman, R. (2011)
\newblock classification methods for functional data.
\newblock In: Ferraty, F. and Romain, Y. (Eds.), The Oxford Handbook of Functional Data Analysis. 
Osford University Press, Oxford, 259--297.

\bibitem{ban03}
Banks, D. and Petricoin, E. (2003).
\newblock Finding cancer signals in mass spectrometry data.
Chance 16, 8--57.

\bibitem{biau_2008} 
Biau, G., Devroye, L. and Lugosi, G. (2008) 
\newblock Consistency of random forests and other averaging classifiers. 
\newblock Journal of Machine  Learning Research, 9, 2015--2033.

 \bibitem{biau_2013}
 Biau. G, Fischer, A. Guedj, B. and Malley, J. (2013)
 \newblock COBRA: A nonlinear aggregation strategy, {\em arXiv:1303.2236}.
 
 \bibitem{biau_2012} 
 Biau, G. (2012)
 \newblock Analysis of a random forests model. 
 \newblock  Journal of Machine Learning Research, 13, 1063--1095.
 
 \bibitem{Bong1_2014} 
 Bongiorno, E., Salinelli, E., Goia, A. and Vieu, P. (2014)
 \newblock An overview of IWFOS'2014. Contributions in infinite-dimensional statistics 
and related topics. In: Enea G. Bongiorno, 
Ernesto Salinelli, Aldo Goia, Philippe Vieu (Eds.), Societ\'a Editrice Esculapio, 1--6.

\bibitem{Bong_2014} 
Contributions in infinite-dimensional statistics and related topics, Societ\'a
Editrice Esculapio (2014) 
\newblock Edited by: Bongiorno, E., Salinelli, E., Goia, A. and Vieu, P.
 
\bibitem{breiman_96} 
Breiman, L. (1996) 
\newblock Bagging predictors. 
\newblock Machine Learning,  24, 123--140.
 
 \bibitem{breiman_98} 
 Breiman, L. (1998) 
 \newblock Arcing classifiers. 
 \newblock The Annals of Statistics, 24, 801--849.
 
  \bibitem{breiman_2001} 
  Breiman, L. (2001) 
  \newblock Random forests.
  \newblock Machine Learning, 45, 5--32.
 
  \bibitem{bunea_2007} 
  Bunea, F., Tsybakov, A. B. and Wegkamp, M. H. (2007)
  \newblock Aggregation for  gaussian regression. 
  \newblock The Annals of Statistics, 35, 1674--1697. 

\bibitem{cues06}
Cuesta-Albertos, J.A., Fraiman, R. and Ransford, T. (2006)
\newblock Random projections and goodness-of-fit tests in infinite-dimensional spaces.
Bulletin of the Brazilian Mathematical Society  37, 1--25.

\bibitem{cuevas_2012} 
Cuevas, A. (2012)
\newblock A partial overview of the theory of statistics with functional data.
\newblock Journal of Statistical Planning and Inference, 147, 1--23. 

\bibitem{fraiman_etal_97} 
Fraiman, R. , Liu, R. and Meloche, J. (1997) 
\newblock Multivariate density estimation by 
probing depth. $L_1$--Statistical Procedures and Related Topics. 
\newblock IMS  Lectures Notes - Monograph series, 31, 415--430.

\bibitem{hall_2005} 
Hall, P. and Samworth, R. (2005) 
\newblock Properties of bagged nearest neighbor classifiers. 
\newblock Journal of the Royal Statistical Society B, 74 (2), 267--286. 

\bibitem{hall_2012} 
Delaigle, A. and Hall, P. (2012) 
\newblock Achieving near perfect classification for 
functional data.  
\newblock Journal of the Royal Statistical Society B, 74 (2), 267--286. 

\bibitem{mojir1} 
Mojirsheibani, M. (1999) 
\newblock Combining classifiers via discretization. 
\newblock Journal of the American Statistical Association, 94, 600--609.

\bibitem{mojir2}
Mojirsheibani, M. (2002) 
\newblock An almost surely optimal combined classification rule. 
\newblock Journal of Multivariate Analysis, 81, 28--46.

\bibitem{yang_2004} 
Yang, Y. (2004)
\newblock Aggregating regression procedures to improve performance.
\newblock Bernoulli, 10, 25--47.
\end{thebibliography}
\end{document}